\documentclass[a4paper,twoside,10pt,leqno]{article}
\usepackage{expl3}
\usepackage{amssymb,amsmath,amsthm}
\usepackage[all]{xy}
\usepackage[nottoc, notlof, notlot]{tocbibind}
\usepackage{stmaryrd}

\usepackage{tikz}
\usetikzlibrary{arrows}

\makeatletter
\def\blfootnote{\gdef\@thefnmark{}\@footnotetext}
\makeatother
\title{On the arc-analytic type of some weighted homogeneous polynomials\blfootnote{2010 Mathematics Subject Classification. 14B05 (32S15 14E18 14M25 14P99).}}
\newcommand\shorttitle{On the arc-analytic type of some weighted homogeneous polynomials}
\author{Jean-Baptiste Campesato\footnote{\parbox[t][2em][s]{\textwidth}{Department of Mathematics, Faculty of Science, Saitama University, 255 Shimo-Okubo, Sakura-ku, Saitama 338-8570, Japan. \newline E-mail address: \url{jbcampesato@mail.saitama-u.ac.jp}\\Research supported by a Japan Society for the Promotion of Science (JSPS) Postdoctoral Fellowship (Short-term) for North American and European Researchers.}}}
\date{December 2, 2017}

\usepackage{ifxetex}
\ifxetex
    \usepackage{polyglossia}
    \setdefaultlanguage[variant=us]{english}
    \usepackage{fontspec}
    \usepackage{unicode-math}
\else
    \usepackage[english]{babel}
    \usepackage{mathrsfs}
    \usepackage{fouriernc}
    
    \DeclareSymbolFont{calletters}{OMS}{cmsy}{m}{n}
    \DeclareSymbolFontAlphabet{\mathcal}{calletters}
    \newcommand{\M}{\mathcal M}
    \usepackage{tgpagella}
    \usepackage[T1]{fontenc}
    \usepackage[utf8x]{inputenc}
\fi

\usepackage{hyperref}
\usepackage{color}
\definecolor{darkred}{rgb}{.5,0,0}
\definecolor{darkgreen}{rgb}{0,.5,0}
\definecolor{darkblue}{rgb}{0,0,.5}
\hypersetup{
    pdfencoding=unicode,
    colorlinks=true,
    linkcolor=darkred,
    citecolor=darkgreen,
    urlcolor=darkblue,
    bookmarks=false,
    pdftoolbar=true,
    pdfmenubar=true,
    pdffitwindow=true,
    pdfstartview={FitH},
    pdfcreator={},
    pdfproducer={},
    pdfnewwindow=true,
}
\makeatletter
    \hypersetup{pdftitle={\shorttitle},pdfauthor={Jean-Baptiste Campesato}}
\makeatother

\usepackage[inner=4cm,outer=3cm,top=2cm,bottom=2cm,includeheadfoot,headsep=.5cm]{geometry}
\usepackage{fancyhdr}
\renewcommand{\thepage}{\arabic{page}}
\fancypagestyle{plain}{
\fancyhf{}

}

\def\footindent{2em}
\makeatletter
\renewcommand\@makefntext[1]{\leftskip=\footindent\hskip-\footindent\@makefnmark#1}
\makeatother
\usepackage{perpage}
\MakePerPage{footnote}
\makeatletter
\ifxetex
\newcommand*{\fnsymbolsingle}[1]{\ensuremath{\ifcase#1\or\star\or\dagger\or\ddagger\or\textsection\or\|\or\textparagraph\or\else\@ctrerr\fi}}
\else
\newcommand*{\fnsymbolsingle}[1]{\ensuremath{\ifcase#1\or\star\or\dagger\or\ddagger\or\mathsection\or\|\or\mathparagraph\or\else\@ctrerr\fi}}
\fi
\makeatother
\usepackage{alphalph}
\newalphalph{\fnsymbolmult}[mult]{\fnsymbolsingle}{}

\makeatletter
\renewcommand{\maketitle}{
  \newpage
  \null
  \thispagestyle{plain}
  \begin{center}
    {\LARGE \@title \par}
    \vskip 1.5em
    {\large
      \lineskip .5em
        \@author
      \par}
    \vskip 1em
    {\large \@date}
  \end{center}
  \par
  \vskip 1.5em}
\makeatother

\theoremstyle{plain}
\newtheorem{thm}{Theorem}[section]
\newtheorem{prop}[thm]{Proposition}
\newtheorem{cor}[thm]{Corollary}
\newtheorem{lemma}[thm]{Lemma}
\theoremstyle{definition}
\newtheorem{defn}[thm]{Definition}

\newtheorem{rem}[thm]{Remark}

\newtheorem{notation}[thm]{Notation}

\usepackage{enumitem}
\setlist[itemize]{labelindent=.6em, itemindent=1em, leftmargin=!, label=\textbullet}

\usepackage{ifmtarg}
\makeatletter
  \newcommand{\TODO}[1]{\@ifmtarg{#1}{\emph{\textbf{TODO}}~}{\emph{\textbf{TODO:}~#1~}}}
\makeatother

\newcommand{\ac}{\operatorname{ac}}

\renewcommand{\d}{\mathrm{d}}
\newcommand{\X}{\mathfrak X}

\newcommand{\mon}{\mathrm{mon}}
\newcommand{\pr}{\operatorname{pr}}

\newcommand{\AS}{\mathcal{AS}}
\newcommand{\mult}{\operatorname{mult}}
\newcommand{\Conv}{\operatorname{Conv}}
\newcommand{\id}{\operatorname{id}}
\renewcommand{\M}{\mathcal M}

\newcommand{\lcm}{\operatorname{lcm}}
\newcommand{\pt}{\mathrm{pt}}
\newcommand{\naive}{\mathrm{naive}}
\newcommand{\cst}{\mathrm{cst}}

\makeatletter
\newcommand{\vast}{\bBigg@{4}}
\newcommand{\Vast}{\bBigg@{5}}
\makeatother

\usepackage{graphicx}
\def\acts{\ensuremath{\rotatebox[origin=c]{-90}{$\circlearrowright$}}}

\begin{document}
\maketitle

\begin{abstract}
It is known that the weights of a complex weighted homogeneous polynomial $f$ with isolated singularity are analytic invariants of $(\mathbb C^d,f^{-1}(0))$. When $d=2,3$ this result holds by assuming merely the topological type instead of the analytic one.

G. Fichou and T. Fukui recently proved the following real counterpart: the blow-Nash type of a real singular non-degenerate convenient weighted homogeneous polynomial in three variables determines its weights.

The aim of this paper is to generalize the above-cited result with no condition on the number of variables. We work with a characterization of the blow-Nash equivalence called the arc-analytic equivalence. It is an equivalence relation on Nash function germs with no continuous moduli which may be seen as a semialgebraic version of the blow-analytic equivalence of T.-C. Kuo.
\end{abstract}

\tableofcontents

\section{Introduction}
K. Saito \cite{Sai71} proved that the weights of a complex weighted homogeneous polynomial $f$ with isolated singularity are analytic invariants of $(\mathbb C^d,f^{-1}(0))$. Then it was proved for $d=2$ and $d=3$, respectively by E. Yoshinaga and M. Suzuki \cite{YS79} and by O. Saeki \cite{Sae88}, that the weights are also determined by the topological type of $(\mathbb C^d,f^{-1}(0))$.

T. Fukui \cite[Conjecture 9.2]{Fuk97} conjectured the real counterpart for the blow-analytic equivalence of T.-C. Kuo \cite{Kuo85}. Do two weighted homogeneous polynomials with isolated singularity which are blow-analytically equivalent share the same weights?

In the two variable case, the positive answer has been given by O. M. Abderrahmane \cite{Abd06} for weighted homogeneous polynomials non-degenerate with respect to their Newton polyhedron.

Then G. Fichou and T. Fukui \cite{FF16} proved the conjecture for non-degenerate convenient weighted homogeneous polynomials in three variables up to the blow-Nash equivalence.

The aim of this article consists in proving that Fukui's conjecture holds for non-degenerate convenient weighted homogeneous polynomials, with no condition on the number of variables, up to the arc-analytic equivalence, a characterization of the blow-Nash equivalence.

It is proved in \cite{JBC2} that the arc-analytic type of a singular Brieskorn polynomial determines its exponents. The main idea of the proof of this last result consists in computing the motivic zeta function of a Brieskorn polynomial thanks to a convolution formula. Next we show that we may recover the exponents from the formula obtained with the previous computation. We conclude by noticing that the motivic zeta function is an invariant of the arc-analytic equivalence.

Notice that Brieskorn polynomials are non-degenerate convenient weighted homogeneous polynomials. However, since the variables are no longer separated, we can't use the convolution formula anymore to recover the weights of such a weighted homogeneous polynomial. Instead, the main idea of this article consists in using a formula allowing one to compute the zeta function of a non-degenerate polynomial from its Newton polyhedron, in order to reduce to a similar combinatorial problem as in the Brieskorn polynomials case. \\

The main theorem of this article states that the arc-analytic type of a singular non-degenerate convenient weighted homogeneous polynomial determines its weights.
\begin{thm}[Main theorem]\label{thm:Main}
Let $f,g\in\mathbb R[x_1,\ldots,x_d]$ be two non-degenerate convenient weighted homogeneous polynomials. If $f$ and $g$ are arc-analytically equivalent, then
\begin{enumerate}[nosep, label=(\roman*)]
\item either they are both non-singular, and in this case they are arc-analytically equivalent to $x_1$,
\item or they share the same weights (up to permutations and up to a positive factor, as explained in Remark \ref{rem:defnW}).
\end{enumerate}
\end{thm}

We briefly recall the definitions of the objects involved in the previous statement.
\begin{rem}
In this article a \emph{non-degenerate} polynomial is a polynomial which is non-degenerate with respect to its Newton polyhedron. See Definition \ref{defn:NDpol}.
\end{rem}

\begin{defn}
A polynomial $f\in\mathbb R[x_1,\ldots,x_d]$ is said to be \emph{convenient} if its Newton polyhedron\footnote{See Definition \ref{defn:NewPol}.} intersects all the coordinate axes, i.e. for all $i=1,\ldots,d$ a pure monomial of the form $x_i^{\delta_i}$, $\delta_i\in\mathbb N_{>0}$, appears in the expansion of $f$ with a nonzero coefficient.
\end{defn}

\begin{defn}
A polynomial $f\in\mathbb R[x_1,\ldots,x_d]$ is said to be \emph{weighted homogeneous} if there exists $(w_1,\ldots,w_d)\in\mathbb N_{>0}^d$ and $w\in\mathbb N_{>0}$ such that $$\forall\lambda\in\mathbb R,\,f(\lambda^{w_1}x_1,\ldots,\lambda^{w_d}x_d)=\lambda^wf(x_1,\ldots,x_d)$$
\end{defn}

\begin{rem}\label{rem:defnW}
Under the assumptions of Theorem \ref{thm:Main}, the weight system $(w_1,\ldots,w_d)\in\mathbb N_{>0}^d$ of $f$ is well defined up to a multiplicative factor. Notice also that reordering the variables doesn't change the arc-analytic type of a polynomial.

Hence, from now on, we refer to the weights of $f$ as the unique primitive vector\footnote{i.e. $\gcd(w_1,\ldots,w_d)=1$.} $(w_1,\ldots,w_d)\in\mathbb N_{>0}^d$ such that $f$ is weighted homogeneous with respect to the weights $(w_1,\ldots,w_d)$ and we assume that $w_1\ge\cdots\ge w_d$.
\end{rem}

This article is structured as follows. First, Section \ref{sect:recol} recalls the definition of the arc-analytic equivalence together with the definition and the needed properties of the invariant of the arc-analytic equivalence which we will use, namely the motivic zeta function. Then, Section \ref{sect:nondegenerate} introduces a formula to compute the above cited invariant for a non-degenerate polynomial. Finally, Section \ref{sect:proof} contains the essence of the proof of Theorem \ref{thm:Main}. \\

\noindent\textbf{Acknowledgements.} I am very grateful to Toshizumi Fukui for his warm welcome in Saitama University and for our fruitful discussions. I would like to thank Goulwen Fichou and Adam Parusiński for their useful comments on a preliminary version of this work.

\section{Recollection}\label{sect:recol}
\subsection{The arc-analytic equivalence}
The arc-analytic equivalence is a characterization of the blow-Nash equivalence of G. Fichou \cite{Fic05,Fic05-bis,Fic06}. It avoids using Nash modifications in its definition and it allows one to prove, as expected, that it is an equivalence relation on Nash function germs.

G. Fichou \cite{Fic05,Fic05-bis} proved that the blow-Nash equivalence admits no continuous moduli for isolated singularities. A. Parusiński and L. Păunescu \cite{PP} recently proved that the arc-analytic equivalence admits no continuous moduli even for families of non-isolated singularities. We refer the reader to the survey \cite[\S7.1]{JBC3} for more details on the arc-analytic equivalence.

\begin{defn}[{\cite[Definition 7.5]{JBC2}}]
Two Nash function germs\footnote{A Nash function is a real analytic function with semialgebraic graph.} $f,g:(\mathbb R^d,0)\rightarrow(\mathbb R,0)$ are \emph{arc-analytically equivalent} if $f=g\circ\varphi$ where
\begin{itemize}[nosep]
\item $\varphi:(\mathbb R^d,0)\rightarrow(\mathbb R^d,0)$ is a semialgebraic homeomorphism,
\item $\varphi$ is arc-analytic\footnote{i.e. $\varphi$ maps real analytic arcs to real analytic arcs by composition, it is a notion defined by K. Kurdyka \cite{Kur88}.},
\item there exists $c>0$ such that $|\det\d \varphi|>c$ where $\d \varphi$ is defined\footnote{K. Kurdyka \cite{Kur88} proved that a semialgebraic arc-analytic map is real analytic outside a set of codimension 2.}.
\end{itemize}
\end{defn}

\begin{rem}[{\cite[Theorem 3.5]{JBC1}}] Let $\varphi:(\mathbb R^d,0)\rightarrow(\mathbb R^d,0)$ be a semialgebraic homeomorphism which is also arc-analytic. If there exists $c>0$ such that  $|\det\d \varphi|>c$ where $\d \varphi$ is defined then $\varphi^{-1}$ is also arc-analytic and there exists $c'>0$ such that $|\det\d \varphi^{-1}|>c'$ where $\d \varphi^{-1}$ is defined.
\end{rem}

\begin{prop}[{\cite[\S7]{JBC2}}]
\begin{itemize}[nosep]
\item The arc-analytic equivalence is an equivalence relation.
\item Two Nash function germs are arc-analytically equivalent if and only if they are blow-Nash equivalent.
\end{itemize}
\end{prop}

\subsection{A motivic invariant of the arc-analytic equivalence}
This section recalls the definition and some properties of the motivic zeta function introduced in \cite{JBC2}. We adopt the notations of the survey \cite[\S7.2]{JBC3}. This motivic zeta function is an invariant of the arc-analytic equivalence whose construction is similar to the motivic zeta functions of Denef--Loeser \cite{DL98} as the ones of Koike--Parusiński \cite{KP03} and Fichou \cite{Fic05,Fic05-bis}.

\begin{prop}[{\cite[\S4.2]{Par04}}]
A semialgebraic subset $S$ of $\mathbb P^n_\mathbb R$ is an $\AS$-set if for every real analytic arc $\gamma:(-1,1)\rightarrow\mathbb P^n_\mathbb R$ satisfying $\gamma((-1,0))\subset S$, there exists $\varepsilon>0$ such that $\gamma((0,\varepsilon))\subset S$.
\end{prop}

\begin{defn}
Let $K_0(\AS)$ be the free abelian group spanned by the symbols\footnote{It is well defined since $\AS$ is a set.} $[X]$ with $X\in\AS$ modulo the following relations:
\begin{enumerate}[label=(\arabic*), nosep]
\item If there is a bijection $X\rightarrow Y$ whose graph is $\AS$ then $[X]=[Y]$.
\item For $X\in\AS$ and $Y\subset X$ a closed $\AS$-subset we set $[X\setminus Y]+[Y]=[X]$.
\end{enumerate}
The cartesian product induces a ring structure:
\begin{enumerate}[label=(\arabic*), nosep, resume]
\item $[X][Y]=[X\times Y]$.
\end{enumerate}
\end{defn}

\begin{rem}
\begin{itemize}[nosep]
\item We denote by $0=[\varnothing]$ the class of the empty set. It is the unit of the addition.
\item We denote by $1=[\pt]$ the class of the point. It is the unit of the multiplication.
\item We denote by $\mathbb L_\AS=[\mathbb R]$ the class of the affine line and we set $\M_\AS=K_0(\AS)\left[\mathbb L_\AS^{-1}\right]$.
\end{itemize}
\end{rem}

\begin{thm}[\cite{MP03},\cite{Fic05},\cite{MP11}]
There is a unique map $\beta:\AS\rightarrow\mathbb Z[u]$, called the \emph{virtual Poincaré polynomial}, such that
\begin{itemize}[nosep]
\item $\beta$ factorises through a ring morphism $\beta:K_0(\AS)\rightarrow\mathbb Z[u]$.
\item If $X$ compact and non-singular then $\beta(X)=\sum_i\dim H_i(X,\mathbb Z_2)u^i$.
\end{itemize}
Moreover, if $X\neq\varnothing$ then $\deg\beta(X)=\dim X$ and the leading coefficient of $\beta(X)$ is positive\footnote{If $X=\varnothing$ then $\beta(\varnothing)=0$.}.
\end{thm}

\begin{rem}
The virtual Poincaré polynomial induces a ring morphism $\beta:\M_\AS\rightarrow\mathbb Z[u,u^{-1}]$.
\end{rem}

The following Grothendieck group is an adaptation of the one of Guibert--Loeser--Merle \cite{GLM} to our settings.
\begin{defn}[{\cite[Definition 3.4]{JBC2}}]\label{defn:K0ASmon}
For $n\in\mathbb N_{>0}$, we denote by $K_0(\AS^n_\mon)$ the free abelian group spanned by the symbols $$\left[\varphi_X:\mathbb R^*\acts X\rightarrow\mathbb R^*\right]$$ where $X\in\AS$, the graph $\Gamma_{\varphi_X}\in\AS$, the graph of the action $\Gamma_{\mathbb R^*\times X\rightarrow X}\in\AS$ and finally for all $(\lambda,x)\in\mathbb R^*\times X$, $\varphi_X(\lambda\cdot x)=\lambda^n\varphi_X(x)$, modulo the following relations
\begin{enumerate}[ref=\ref{defn:K0ASmon}.(\arabic*), label=(\arabic*), nosep]
\item If there exists $f:X\rightarrow Y$ an $\mathbb R^*$-equivariant bijection with $\AS$-graph such that the following diagram commutes $$\xymatrix{X \ar[rr]^f_\simeq \ar[rd]_{\varphi_X} & & Y \ar[dl]^{\varphi_Y} \\ & \mathbb R^* &}$$ then we set $$\left[\varphi_X:\mathbb R^*\acts X\rightarrow\mathbb R^*\right]=\left[\varphi_Y:\mathbb R^*\acts Y\rightarrow\mathbb R^*\right]$$
\item If $Y$ is an $\mathbb R^*$-invariant closed $\AS$-subset of $X$ then $$\left[\varphi_X:\mathbb R^*\acts X\rightarrow\mathbb R^*\right]=\left[\varphi_{X|Y}:\mathbb R^*\acts Y\rightarrow\mathbb R^*\right]+\left[\varphi_{X|X\setminus Y}:\mathbb R^*\acts X\setminus Y\rightarrow\mathbb R^*\right]$$
\item\label{item:K0liftings} Let $\varphi_Y:\mathbb R^*\acts_\tau Y\rightarrow\mathbb R^*$ be a symbol and $\psi=\varphi_Y\pr_Y:Y\times\mathbb R^m\rightarrow\mathbb R^*$. Let $\sigma$ and $\sigma'$ be two actions of $\mathbb R^*$ on $Y\times\mathbb R^m$ which are two liftings\footnote{i.e. $\pr_Y(\lambda\cdot_\sigma x)=\lambda\cdot_\tau\pr_Y(x)$.} of $\tau$ then $\psi:\mathbb R^*\acts_{\sigma}(Y\times\mathbb R^m)\rightarrow\mathbb R^*$ and $\psi:\mathbb R^*\acts_{\sigma'}(Y\times\mathbb R^m)\rightarrow\mathbb R^*$ are two symbols and we add the relation $$\left[\psi:\mathbb R^*\acts_{\sigma}(Y\times\mathbb R^m)\rightarrow\mathbb R^*\right]=\left[\psi:\mathbb R^*\acts_{\sigma'}(Y\times\mathbb R^m)\rightarrow\mathbb R^*\right]$$
\end{enumerate}
The fiber product over $\mathbb R^*$ induces a ring structure:
\begin{enumerate}[label=(\arabic*), nosep, resume]
\item We add the relation $$\left[\varphi_X:\mathbb R^*\acts X\rightarrow\mathbb R^*\right]\left[\varphi_Y:\mathbb R^*\acts Y\rightarrow\mathbb R^*\right]=\left[\mathbb R^*\acts(X\times_{\mathbb R^*}Y)\rightarrow\mathbb R^*\right]$$ where the action of $\mathbb R^*$ on $X\times_{\mathbb R^*}Y$ is diagonal from the previous ones.
\end{enumerate}
The cartesian product induces a structure of $K_0(\AS)$-algebra\footnote{The algebra structure is given by the structural morphism $K_0(\AS)\rightarrow K_0(\AS^n_\mon)$ defined by $[A]\rightarrow\left[\pr_{\mathbb R^*}:A\times\mathbb R^*\rightarrow\mathbb R^*\right]$ where the $\mathbb R^*$-action is $\lambda\cdot(a,r)=(a,\lambda^nr)$.}:
\begin{enumerate}[label=(\arabic*), nosep, resume]
\item Let $[A]\in K_0(\AS)$ and $\left[\varphi_X:\mathbb R^*\acts X\rightarrow\mathbb R^*\right]\in K_0(\AS^n_\mon)$ then we set $$[A]\cdot\left[\varphi_X:\mathbb R^*\acts X\rightarrow\mathbb R^*\right]=\left[\varphi_X\pr_X:A\times X\rightarrow\mathbb R^*\right]$$ where the action is trivial on $A$ and unchanged on $X$.
\end{enumerate}
\end{defn}

\begin{defn}
%The order $n\prec m\Leftrightarrow\exists k\in\mathbb N_{>0},\,n=km$ together with the morphisms $\theta_n^m:K_0(\AS^m_\mon)\rightarrow K_0(\AS^n_\mon)$ modifying the $\mathbb R^*$-action by $\lambda\cdot_{K_0(\AS^n_\mon)}x=\lambda^k\cdot_{K_0(\AS^m_\mon)}x$ for $n\prec m$ define a direct system and we set $K_0(\AS_\mon)=\varinjlim K_0(\AS^n_\mon)$. Notice that $K_0(\AS_\mon)$ has a natural structure of $K_0(\AS)$-algebra. 
We set $K_0(\AS_\mon)=\varinjlim K_0(\AS^n_\mon)$ where the direct system is induced by modifying the $\mathbb R^*$-action by $\lambda\cdot_{K_0(\AS^n_\mon)}x=\lambda^k\cdot_{K_0(\AS^m_\mon)}x$ when $n=km$. Notice that $K_0(\AS_\mon)$ has a natural structure of $K_0(\AS)$-algebra.
\end{defn}

\begin{rem}
\begin{itemize}[nosep]
\item We denote by $0=[\varnothing]$ the class of the empty set. It is the unit of the addition.
%\item For $n\in\mathbb N_{>0}$, we set $\mathbb1_n=\left[\id:\mathbb R^*\acts\mathbb R^*\rightarrow\mathbb R^*\right]\in K_0(\AS^n_\mon)$ where the action is defined by $\lambda\cdot r=\lambda^nr$. Notice that these elements are compatible under the direct limit so that they induce an element $\mathbb 1\in K_0(\AS_\mon)$. Indeed, for all $m,n\in\mathbb N_{>0}$, $m\prec 1$, $n\prec 1$, $\mathbb 1_n=\theta_n^1(\mathbb 1_1)$ and $\mathbb 1_m=\theta_m^1(\mathbb 1_1)$. Hence, we will use the previous identification to identify $\mathbb 1$ with $\mathbb 1_1$ and assume that the action is given by translation. It is the unit of the multiplication.
\item For $n\in\mathbb N_{>0}$, we set $\mathbb1_n=\left[\id:\mathbb R^*\acts\mathbb R^*\rightarrow\mathbb R^*\right]\in K_0(\AS^n_\mon)$ where the action is defined by $\lambda\cdot r=\lambda^nr$, it is the unit of the product. Notice that the $\mathbb 1_n$ are identified under the inductive limit so that they induce an element $\mathbb 1\in K_0(\AS_\mon)$ which is also the unit of the product. Under the previous identification, we may chose $\mathbb 1_1$ as a representative and assume that $\mathbb1=\left[\id:\mathbb R^*\acts\mathbb R^*\rightarrow\mathbb R^*\right]$ where the action is given by translation i.e. $\lambda\cdot r=\lambda r$.
\item We denote by $\mathbb L=\mathbb L_{\AS}\cdot\mathbb1$ the class of the affine line and we set $\M=K_0(\AS_\mon)\left[\mathbb L^{-1}\right]$ so that $\M$ has a natural structure of $\M_{\AS}$-algebra.
\end{itemize}
\end{rem}

The following properties of $K_0(\AS_\mon)$ and $\M_\AS$ will be useful in what follows.
\begin{prop}[{\cite[End of \S3]{JBC2}}]
The map $\AS_\mon^n\rightarrow\AS$ defined by $\left(\varphi_X:\mathbb R^*\acts X\rightarrow\mathbb R^*\right)\mapsto X$ induces:
\begin{itemize}[nosep]
\item A morphism of $K_0(\AS)$-modules $\overline{\vphantom{1em}\ \cdot\ }:K_0(\AS_\mon)\rightarrow K_0(\AS)$,
\item A morphism of $\M_{\AS}$-modules $\overline{\vphantom{1em}\ \cdot\ }:\M\rightarrow\M_{\AS}$.
\end{itemize}
\end{prop}

\begin{prop}[{\cite[Proposition 4.16]{JBC2}}]
Let $\varepsilon\in\{+,-\}$. The map $\AS_\mon^n\rightarrow\AS$ defined by $\left(\varphi_X:\mathbb R^*\acts X\rightarrow\mathbb R^*\right)\mapsto\varphi_X^{-1}(\varepsilon1)$ induces:
\begin{itemize}[nosep]
\item A morphism of $K_0(\AS)$-algebras $F^\varepsilon:K_0(\AS_\mon)\rightarrow K_0(\AS)$,
\item A morphism of $\M_{\AS}$-algebras $F^\varepsilon:\M\rightarrow\M_{\AS}$.
\end{itemize}
\end{prop}

\begin{rem}
The morphisms of the last proposition are compatible with the ring structures since the fiber product over a point coincides with the cartesian product.
\end{rem}

The following zeta function is a real analog of the equivariant motivic zeta function of Denef--Loeser.
\begin{defn}[{\cite[Definition 4.2]{JBC2}}]
Let $f:(\mathbb R^d,0)\rightarrow(\mathbb R,0)$ be a Nash function germ. We set $$Z_f(T)=\sum_{n\ge1}\left[\ac_f^n:\mathbb R^*\acts\X_n(f)\rightarrow\mathbb R^*\right]\mathbb L^{-nd}T^n\in\M\llbracket T\rrbracket$$ where $\X_n(f)=\left\{\gamma=a_1t+\cdots+a_nt^n,\,a_i\in\mathbb R^d,\,f\gamma(t)=ct^n+\cdots,\,c\neq0\right\}$, $\ac_f^n:\X_n(f)\rightarrow\mathbb R^*$ is the angular component map defined by $\ac_f^n(\gamma)=\ac(f\gamma)=c$ and where the action is defined by $\lambda\cdot\gamma(t)=\gamma(\lambda t)$.
\end{defn}

\begin{rem}
Notice that the coefficients $P_i(a_1,\ldots,a_n)$ of $$f(a_1t+\cdots+a_nt^n)=P_1(a_1,\ldots,a_n)t+\cdots+P_n(a_1,\ldots,a_n)t^n+\cdots$$ are polynomials in the $a_{ij}$. So $\X_n(f)$ is the Zariski-constructible subset of $\mathbb R^{dn}$ described by $P_1=\cdots=P_{n-1}=0$ and $P_n\neq0$ and $\X_n(f)$ is in $\AS$. Under the same identification, $\ac_f^n$ is given by the polynomial $P_n$ so that its graph is in $\AS$.
\end{rem}

\begin{thm}[{\cite[Theorem 7.11]{JBC2}}]
If $f,g:(\mathbb R^d,0)\rightarrow(\mathbb R,0)$ are two arc-analytically equivalent Nash function germs then $Z_f(T)=Z_g(T)$.
\end{thm}

\begin{defn}[{\cite[Definition 6.6]{JBC2}}]
Let $f:(\mathbb R^d,0)\rightarrow(\mathbb R,0)$ be a Nash function germ. We define the modified zeta function of $f$ by $$\tilde Z_f(T)=Z_f(T)-\frac{\mathbb 1-Z_f^{\naive}(T)}{\mathbb 1-T}+\mathbb1$$ where $Z_f^{\naive}(T)$ is defined by applying $\alpha\mapsto\overline{\alpha}\cdot\mathbb1$ to the coefficients of $Z_f(T)$.
\end{defn}

\begin{rem}
Particularly, the modified zeta function is an invariant of the arc-analytic equivalence too. Moreover, by \cite[Corollary 6.14]{JBC2}, it encodes the same information as $Z_f(T)$:
$$Z_f(T)=\tilde Z_f(T)+\frac{\mathbb 1-\mathbb L^{-1}\tilde Z_f^{\naive}(T)}{\mathbb 1-\mathbb L^{-1}T}-\mathbb1$$
\end{rem}

\section{The modified zeta function of a non-degenerate polynomial}\label{sect:nondegenerate}
\begin{defn}\label{defn:NewPol}
Let $f=\sum_{\nu\in\mathbb N^d}c_\nu x^\nu\in\mathbb R[x_1,\ldots,x_d]$. We define the Newton polyhedron of $f$ by $$\Gamma_f=\Conv\left(\bigcup_{\substack{\nu\in\mathbb N^d\\c_\nu\neq0}}\left(\nu+\mathbb R_{\ge0}^d\right)\right)$$
\end{defn}

\begin{defn}
For $\tau$ a face of $\Gamma_f$, we set $f_\tau(x)=\sum_{\nu\in\tau}c_\nu x^\nu$.
\end{defn}

\begin{defn}
We define the supporting function $m:\mathbb R_{\ge0}^d\rightarrow\mathbb R_{\ge0}$ by $m(k)=\inf\left\{k\cdot x,\,x\in\Gamma_f\right\}$.
\end{defn}

\begin{defn}
The trace of $k\in\mathbb R_{\ge0}^d$ is $\tau(k)=\left\{x\in\Gamma_f,\,k\cdot x=m(k)\right\}$.
\end{defn}

\begin{defn}
We define the dual cone of a face $\tau$ of $\Gamma_f$ by $\sigma(\tau)=\left\{k\in\mathbb R_{\ge0}^d,\,\tau(k)=\tau\right\}$.
\end{defn}

\begin{notation}
We denote by $\Gamma_f^c$ the set of compact faces of $\Gamma_f$.
\end{notation}

\begin{defn}\label{defn:NDpol}
A polynomial $f\in\mathbb R[x_1,\ldots,x_d]$ is non-degenerate (with respect to its Newton polyhedron) if $$\forall\tau\in\Gamma_f^c,\,\left\{x\in(\mathbb R^*)^d,\,\forall i=1,\ldots,d,\,\frac{\partial f_\tau}{\partial x_i}(x)=0\right\}=\varnothing$$
\end{defn}

\begin{rem}
For a face $\tau\in\Gamma_f^c$, notice that $f_\tau$ is weighted homogeneous and hence, by Euler formula, $$\left\{x\in(\mathbb R^*)^d,\,\frac{\partial f_\tau}{\partial x_1}(x)=\cdots=\frac{\partial f_\tau}{\partial x_d}(x)=0\right\}=\left\{x\in(\mathbb R^*)^d,\,f_\tau(x)=\frac{\partial f_\tau}{\partial x_1}(x)=\cdots=\frac{\partial f_\tau}{\partial x_d}(x)=0\right\}$$
\end{rem}

\begin{lemma}[{\cite[Proposition 3.13]{Rai12}}]
For $\tau\in\Gamma_f^c$ and $k\in\sigma(\tau)\cap\mathbb N^d$, we set $$\left[f_\tau:(\mathbb R^*)^d\setminus f_\tau^{-1}(0)\rightarrow\mathbb R^*\right]:=\left[f_\tau:\mathbb R^*\acts\left((\mathbb R^*)^d\setminus f_\tau^{-1}(0)\right)\rightarrow\mathbb R^*\right]\in K_0\left(\AS_\mon^{m(k)}\right)$$ where the action is given by $\lambda\cdot(x_i)_i=(\lambda^{k_i}x_i)_i$. \\
Then $\left[f_\tau:(\mathbb R^*)^d\setminus f_\tau^{-1}(0)\rightarrow\mathbb R^*\right]\in K_0(\AS_\mon)$ doesn't depend on the choice of $k\in\sigma(\tau)\cap\mathbb N^d$.
\end{lemma}

\begin{notation}
For $m\in\mathbb Z$, we denote by $\mathcal F^m\M$ the subgroup of $\M$ spanned by the elements of the form $[S]\mathbb L^{-i}$ where $i-\dim S\ge m$. It defines a filtration and we denote by $\widehat\M$ the completion of $\M$ with respect to this filtration.
\end{notation}

\begin{rem}
Notice that $\beta\overline{\vphantom{1em}\ \cdot\ }$ and $\beta F^\pm$ factorises through the image of $\M\rightarrow\widehat\M$ since the kernel is $\cap\mathcal F^m\M$ and if $\alpha\in\cap\mathcal F^m\M$ then $\beta(\overline{\alpha})=0$ and $\beta F^\varepsilon(\alpha)=0$ by consideration on the degree. We deduce that it also holds for the Euler characteristic with compact support $\chi_c$ by evaluating the virtual Poincaré polynomials in $u=-1$.
\end{rem}

A similar formula to the one of the following result has been proved in different settings: by A. N. Varchenko \cite{Var76} for the zeta function of the monodromy, by J. Denef and F. Loeser \cite[\S5]{DL92} for the topological zeta function, by J. Denef and K. Hoornaert \cite[Theorem 4.2]{DH01} for the Igusa $p$-adic zeta function and by G. Guibert for the motivic zeta function \cite[Proposition 2.1.3]{Gui02}. The proof of G. Fichou and T. Fukui \cite{FF16} already relies on an adaptation of this construction to the virtual Poincaré polynomial.

\begin{thm}[{\cite[Theorem 5.15]{JBC2}}]\label{thm:nondegZfclassique}
Let $f\in\mathbb R[x_1,\ldots,x_d]$ be non-degenerate, then the following equality holds in $\widehat\M\llbracket T\rrbracket$:
$$Z_f(T)=\sum_{\tau\in\Gamma_f^c}\left(\left[f_\tau:(\mathbb R^*)^d\setminus f_\tau^{-1}(0)\rightarrow\mathbb R^*\right]+\left[\pr_2:f_\tau^{-1}(0)\cap(\mathbb R^*)^d\times\mathbb R^*\rightarrow\mathbb R^*\right]\frac{\mathbb L^{-1}T}{\mathbb 1-\mathbb L^{-1}T}\right)S_{\sigma(\tau)}(T)$$
where $\mathbb R^*$ acts on $f_\tau^{-1}(0)\cap(\mathbb R^*)^d\times\mathbb R^*$ by $\lambda\cdot(x,t)=(x,\lambda t)$ and where $S_{\sigma(\tau)}(T)=\sum_{k\in\sigma(\tau)\cap\mathbb N^d}\mathbb L^{-|k|}T^{m(k)}$ with $|k|=\sum_{i=1}^dk_i$.
\end{thm}

\begin{rem}
Notice that for the purpose of the previous theorem, we may have simply set $\widehat\M=\M\left[\frac{1}{\mathbb 1-\mathbb L^{-a}},\,a\in\mathbb N_{>0}\right]$.
\end{rem}

The proof of Theorem \ref{thm:Main} relies on the formula induced by Theorem \ref{thm:nondegZfclassique} for the modified zeta function.
\begin{cor}\label{cor:nondegZfmod}
Let $f$ be non-degenerate, then the following equality holds in $\widehat\M\llbracket T\rrbracket$:
$$\hspace{-1cm}\tilde Z_f(T)=\sum_{\tau\in\Gamma_f^c}\left(\left[f_\tau:(\mathbb R^*)^d\setminus f_\tau^{-1}(0)\rightarrow \mathbb R^*\right]-\left[\pr_2:f_\tau^{-1}(0)\cap(\mathbb R^*)^d\times\mathbb R^*\rightarrow\mathbb R^*\right]+\frac{(\mathbb L-\mathbb 1)^d}{\mathbb 1-T}\right)S_{\sigma(\tau)}(T)-\frac{T}{\mathbb1-T}$$
\end{cor}

\begin{rem}
By the very definition of $\mathbb L$ and $\mathbb 1$, we have $$(\mathbb L-\mathbb1)^d=\left[\pr_2:\mathbb R^*\acts(\mathbb R^*)^d\times\mathbb R^*\rightarrow\mathbb R^*\right]$$
\end{rem}

\begin{proof}[Proof of Corollary \ref{cor:nondegZfmod}]
\begin{align*}
Z_f^{\naive}(T)&=\sum_{\tau\in\Gamma_f^c}\left(\overline{\left[(\mathbb R^*)^d\setminus f_\tau^{-1}(0)\right]}\cdot\mathbb1+\overline{\left[f_\tau^{-1}(0)\cap(\mathbb R^*)^d\right]}\cdot(\mathbb L-\mathbb 1)\frac{\mathbb L^{-1}T}{\mathbb 1-\mathbb L^{-1}T}\right)S_{\sigma(\tau)}(T) \\
&=\sum_{\tau\in\Gamma_f^c}\left((\mathbb L-\mathbb 1)^d-\overline{\left[f_\tau^{-1}(0)\cap(\mathbb R^*)^d\right]}\cdot\mathbb1+\overline{\left[f_\tau^{-1}(0)\cap(\mathbb R^*)^d\right]}\cdot(\mathbb L-\mathbb 1)\frac{\mathbb L^{-1}T}{\mathbb 1-\mathbb L^{-1}T}\right)S_{\sigma(\tau)}(T) \\
&=\sum_{\tau\in\Gamma_f^c}\left((\mathbb L-\mathbb 1)^d+\left[\pr_2:f_\tau^{-1}(0)\cap(\mathbb R^*)^d\times\mathbb R^*\rightarrow\mathbb R^*\right]\frac{T-\mathbb 1}{\mathbb 1-\mathbb L^{-1}T}\right)S_{\sigma(\tau)}(T) \\
\end{align*}
Thus
\begin{align*}
\tilde Z_f(T)&=Z_f(T)-\frac{Z_f^{\naive}(T)-\mathbb1}{T-\mathbb1}+\mathbb 1\\
&=\sum_{\tau\in\Gamma_f^c}\vast(\left[f_\tau:(\mathbb R^*)^d\setminus f_\tau^{-1}(0)\rightarrow\mathbb R^*\right]+\left[\pr_2:f_\tau^{-1}(0)\cap(\mathbb R^*)^d\times\mathbb R^*\rightarrow\mathbb R^*\right]\frac{\mathbb L^{-1}T}{\mathbb 1-\mathbb L^{-1}T}\\
&\quad\quad\quad+\frac{(\mathbb L-\mathbb 1)^d}{\mathbb 1-T}-\frac{\left[\pr_2:f_\tau^{-1}(0)\cap(\mathbb R^*)^d\times\mathbb R^*\rightarrow\mathbb R^*\right]}{\mathbb 1-\mathbb L^{-1}T}\vast)S_{\sigma(\tau)}(T)-\frac{T}{\mathbb 1-T}\\
&=\sum_{\tau\in\Gamma_f^c}\left(\left[f_\tau:(\mathbb R^*)^d\setminus f_\tau^{-1}(0)\rightarrow\mathbb R^*\right]-\left[\pr_2:f_\tau^{-1}(0)\cap(\mathbb R^*)^d\times\mathbb R^*\rightarrow\mathbb R^*\right]+\frac{(\mathbb L-\mathbb 1)^d}{\mathbb 1-T}\right)S_{\sigma(\tau)}(T)\\
&\quad\quad-\frac{T}{\mathbb1-T}
\end{align*}
\end{proof}

\section{Application to convenient non-degenerate weighted homogeneous polynomials}\label{sect:proof}
Throughout this section, we use the notation $\tilde Z_f(T)=A_f(T)+B_f(T)-\frac{T}{\mathbb1-T}$ where $$A_f(T)=\sum_{\tau\in\Gamma_f^c}\left(\left[f_\tau:(\mathbb R^*)^d\setminus f_\tau^{-1}(0)\rightarrow\mathbb R^*\right]-\left[\pr_2:f_\tau^{-1}(0)\cap(\mathbb R^*)^d\times\mathbb R^*\rightarrow\mathbb R^*\right]\right)S_{\sigma(\tau)}(T)$$ and $$B_f(T)=\frac{(\mathbb L-\mathbb 1)^d}{\mathbb 1-T}\sum_{\tau\in\Gamma_f^c}S_{\sigma(\tau)}(T)$$

We will see that the modified zeta function of a convenient non-degenerate weighted homogeneous polynomial is very similar to the one obtained for a Brieskorn polynomial \cite[\S8]{JBC2} using the convolution formula.

\begin{rem}\label{rem:simpcones}
When $f$ is a convenient non-degenerate weighted homogeneous polynomial, the cones dual to the compact faces of its Newton polyhedron are simplicial. \\
More precisely such a cone is spanned by $w$ and at most $d-1$ vectors of the natural basis, where $w$ is the primitive vector normal to the compact facet. Indeed, a compact face of the Newton polyhedron of $f$ is the convex hull of $x_i^{\delta_i}$ for $i$ in a subset $I$ of $\{1,\ldots,d\}$ (where we identify the monomials and the exponents). The dual cone of such a face is then spanned by $w$ and $e_i$ for $i\notin I$.
\end{rem}

\begin{notation}
For $I\subset\{1,\ldots,d\}$ we set $\delta_I=\lcm\left\{\delta_i,\,i\in I\right\}$.
\end{notation}

\begin{rem}\label{rem:wvector}
The vector $w$ defined in Remark \ref{rem:simpcones} as the primitive vector normal to the unique compact facet of $\Gamma_f$ is the weight vector of $f$. It is given by $$w=\left(\frac{\delta_{\{1,\ldots,d\}}}{\delta_i}\right)_{i=1,\ldots,d}$$
\end{rem}

\subsection{Computation of $B_f(T)$}
\begin{lemma}\label{lem:Bf}
Let $f\in\mathbb R[x_1,\ldots,x_d]$ be a convenient non-degenerate weighted homogeneous polynomial. We denote by $\delta_i$ the exponent of the pure monomial $x_i$ of $f$. Then
$$B_f(T)-\frac{T}{\mathbb1-T}=-\sum_{m\ge1}\mathbb L^{-\sum_{i=1}^d\left\lfloor\frac{m}{\delta_i}\right\rfloor}T^m$$
\end{lemma}
\begin{proof}
Let $k\in(\mathbb N\setminus\{0\})^d$. Since $(\mathbb R\setminus\{0\})^d=\bigsqcup_{\tau\in\Gamma_f^c}\sigma(\tau)$, there is a unique $\tau\in\Gamma_f^c$ such that $k\in\sigma(\tau)\cap\mathbb N^d$. By definition of $\sigma(\tau)$, $m(k)=k\cdot x$ for $x\in\tau$. By Remark \ref{rem:simpcones}, at least one pure monomial $x_i^{\delta_i}$ lies in $\tau$ and then $m(k)=k\cdot(0,\ldots,\delta_i,\ldots,0)$.

Thus the coefficient of $T^m$ in $B_f(T)-\frac{T}{\mathbb1-T}$ is
\begin{align*} 
(\mathbb L-\mathbb 1)^d\sum_{\substack{k\in(\mathbb N\setminus\{0\})^d\\m(k)\le m}}\mathbb L^{-|k|}-\mathbb 1 &=(\mathbb L-\mathbb 1)^d\sum_{k\in(\mathbb N\setminus\{0\})^d}\mathbb L^{-|k|}-(\mathbb L-\mathbb 1)^d\sum_{\substack{k\in(\mathbb N\setminus\{0\})^d\\m(k)>m}}\mathbb L^{-|k|} -\mathbb1\\
&=\mathbb 1-(\mathbb L-\mathbb 1)^d\prod_{i=1}^d\sum_{k_i\ge\left\lfloor\frac{m}{\delta_i}\right\rfloor+1}\mathbb L^{-k_i}-\mathbb1 \\
&=-(\mathbb L-\mathbb 1)^d\prod_{i=1}^d\frac{\mathbb L^{-\left\lfloor\frac{m}{\delta_i}\right\rfloor-1}}{\mathbb 1-\mathbb L^{-1}} \\
&=-(\mathbb L-\mathbb 1)^d\frac{\prod_{i=1}^d\mathbb L^{-\left\lfloor\frac{m}{\delta_i}\right\rfloor}}{(\mathbb L-\mathbb 1)^d}
\end{align*}
\end{proof}

\begin{cor}\label{cor:nonmult}
Let $f\in\mathbb R[x_1,\ldots,x_d]$ be a convenient non-degenerate weighted homogeneous polynomial. We denote by $\delta_i$ the exponent of the pure monomial $x_i$ of $f$. \\
Denote by $a_n$ the coefficients of $\tilde Z_f(T)$, so that $$\tilde Z_f(T)=\sum_{n\ge1}a_nT^n$$
Let $m\in\mathbb N_{>0}$ such that $\forall i,\,\delta_i\nmid m$. Then $$a_m=-\mathbb L^{-\sum_{i=1}^d\left\lfloor\frac{m}{\delta_i}\right\rfloor}$$
\end{cor}
\begin{proof} 
Under our assumptions $m(k)=k\cdot(0,\ldots,\delta_i,\ldots,0)$ for some $i$. Since $\forall i,\,\delta_i\nmid m$ we get that $A_f(T)$ doesn't contribute to the coefficient of degree $m$.
\end{proof}

\subsection{Computation of $A_f(T)$}
\begin{lemma}\label{lem:Stau}
Let $f\in\mathbb R[x_1,\ldots,x_d]$ be a convenient non-degenerate weighted homogeneous polynomial. For $\varnothing\neq I\subset\{1,\ldots,d\}$ we denote by $\tau_I$ the face of $\Gamma_f^c$ intersecting the $x_i$-axis exactly for $i\in I$. Then
$$S_{\sigma(\tau_I)}(T)=\frac{\mathbb 1}{(\mathbb L-\mathbb 1)^{d-|I|}}\sum_{r\ge1}\mathbb L^{-\sum_{i=1}^d\left\lfloor\frac{r\delta_I}{\delta_i}\right\rfloor}T^{r\delta_I}$$
\end{lemma}

\begin{proof}
Notice that, by Remark \ref{rem:simpcones}, $\sigma(\tau_I)$ is the cone spanned by $w$ and $e_i$ for $i\notin I$. \\
Denote by $a_m$ the coefficients of $S_{\sigma(\tau_I)}$ so that $S_{\sigma(\tau_I)}(T)=\sum_{m\ge1}a_mT^m$. \\
Fix $m\in\mathbb N\setminus\{0\}$. We may assume that $\delta_I|m$ otherwise $a_m=0$. \\
Let $k\in(\mathbb N\setminus\{0\})^d$. Then $k\in\sigma(\tau_I)$ and satisfies $m(k)=m$ if and only if for $i\in I$, $k_i\delta_i=m$ and for $j\notin I$, $k_j\ge\left\lfloor\frac{m}{\delta_j}\right\rfloor+1$. \\
Thus 
\begin{align*}
a_m&=\sum_{\substack{k_i\delta_i=m,i\in I\\k_j\ge\left\lfloor\frac{m}{\delta_j}\right\rfloor+1,\,j\notin I}}\mathbb L^{-|k|}\\
&=\mathbb L^{-\sum_{i\in I}\frac{m}{\delta_i}}\sum_{\substack{k_j\ge\left\lfloor\frac{m}{\delta_j}\right\rfloor+1\\j\notin I}}\mathbb L^{-k_j}\\
&=\mathbb L^{-\sum_{i\in I}\frac{m}{\delta_i}}\prod_{j\notin I}\frac{\mathbb L^{-\left\lfloor\frac{m}{\delta_j}\right\rfloor-1}}{\mathbb 1-\mathbb L^{-1}}\\
&=\frac{\mathbb L^{-\sum_{i=1}^d\left\lfloor\frac{m}{\delta_i}\right\rfloor}}{(\mathbb L-\mathbb1)^{d-|I|}}
\end{align*}
\end{proof}

Fix $I\subset\{1,\ldots,d\}$. Notice that the face $\tau_I$ is entirely included in the intersection of the coordinate hyperplanes $\{x_j=0\}$ for $j\notin I$. Hence, $f_{\tau_I}\in\mathbb R[x_i,i\in I]$, i.e. the variables $x_j,j\notin I$, don't appear in the expansion of $f_{\tau_I}$. Hence, by \ref{item:K0liftings}, we have the following equality
$$\left[f_{\tau_I}:(\mathbb R^*)^d\setminus f_{\tau_I}^{-1}(0)\rightarrow\mathbb R^*\right]=(\mathbb L-\mathbb1)^{d-|I|}\left[f_{\tau_I}:(\mathbb R^*)^{|I|}\setminus f_{\tau_I}^{-1}(0)\rightarrow\mathbb R^*\right]$$
where, in the RHS, we identify $(\mathbb R^*)^{|I|}$ with $\left\{(x_1,\ldots,x_d)\in\mathbb R^d,\,\forall j\notin I,x_j=0,\,\forall i\in I,x_i\neq0\right\}$.

Using the structure of $\M_\AS$-algebra, we also have $$\left[\pr_2:f_{\tau_I}^{-1}(0)\cap(\mathbb R^*)^d\times\mathbb R^*\rightarrow\mathbb R^*\right]=(\mathbb L-\mathbb1)^{d-|I|}\left[\pr_2:f_{\tau_I}^{-1}(0)\cap(\mathbb R^*)^{|I|}\times\mathbb R^*\rightarrow\mathbb R^*\right]$$

We derive the following corollary from Lemma \ref{lem:Stau} and these last two equalities.

\begin{cor}\label{cor:ComputeAf}
Let $f\in\mathbb R[x_1,\ldots,x_d]$ be a convenient non-degenerate weighted homogeneous polynomial. With the previous notations, we obtain
$$\hspace{-1cm}A_f(T)=\sum_{\varnothing\neq I\subset\{1,\ldots,d\}}\left(\left[f_{\tau_I}:(\mathbb R^*)^{|I|}\setminus f_{\tau_I}^{-1}(0)\rightarrow\mathbb R^*\right]-\left[\pr_2:f_{\tau_I}^{-1}(0)\cap(\mathbb R^*)^{|I|}\times\mathbb R^*\rightarrow\mathbb R^*\right]\right)\sum_{r\ge1}\mathbb L^{-\sum_{i=1}^d\left\lfloor\frac{r\delta_I}{\delta_i}\right\rfloor}T^{r\delta_I}$$
\end{cor}

\subsection{Digression: a rationality formula}
It is known that we may deduce from Theorem \ref{thm:nondegZfclassique} a rationality formula for $Z_f(T)$ (or $\tilde Z_f(T)$) by using simplicial (or regular) subdivisions of the dual cones. However, the formulae previously obtained in this section allow us to give a rationality formula for $\tilde Z_f(T)$ without using a subdivision when $f$ is a convenient non-degenerate weighted homogeneous polynomial.

Indeed, for $\varnothing\neq I\subset\{1,\ldots,d\}$ and by writing the Euclidean division of $r$ by $\frac{\delta_{\{1,\ldots,d\}}}{\delta_I}$, we get
\begin{align*}
\sum_{r\ge1}\mathbb L^{-\sum_{i=1}^d\left\lfloor\frac{r\delta_I}{\delta_i}\right\rfloor}T^{r\delta_I}&=\sum_{q\ge0}\sum_{s=0}^{\frac{\delta_{\{1,\ldots,d\}}}{\delta_I}-1}\mathbb L^{-\sum_{i=1}^d\left(q\frac{\delta_{\{1,\ldots,d\}}}{\delta_i}+\left\lfloor\frac{s\delta_I}{\delta_i}\right\rfloor\right)}T^{q\delta_{\{1,\ldots,d\}}+s\delta_I}-\mathbb1 \\
&=\frac{\sum_{s=0}^{\frac{\delta_{\{1,\ldots,d\}}}{\delta_I}-1}\mathbb L^{-\sum_{i=1}^d\left\lfloor\frac{s\delta_I}{\delta_i}\right\rfloor}T^{s\delta_I}}{\mathbb 1-\mathbb L^{-\sum_{i=1}^d\frac{\delta_{\{1,\ldots,d\}}}{\delta_i}}T^{\delta_{\{1,\ldots,d\}}}}-\mathbb1 \\
&=\frac{\sum_{s=1}^{\frac{\delta_{\{1,\ldots,d\}}}{\delta_I}}\mathbb L^{-\sum_{i=1}^d\left\lfloor\frac{s\delta_I}{\delta_i}\right\rfloor}T^{s\delta_I}}{\mathbb 1-\mathbb L^{-\sum_{i=1}^d\frac{\delta_{\{1,\ldots,d\}}}{\delta_i}}T^{\delta_{\{1,\ldots,d\}}}} \\
&=\frac{\sum_{s=1}^{\frac{\delta_{\{1,\ldots,d\}}}{\delta_I}}\mathbb L^{-\sum_{i=1}^d\left\lfloor\frac{s\delta_I}{\delta_i}\right\rfloor}T^{s\delta_I}}{\mathbb 1-\mathbb L^{-|w|}T^{m(w)}}
\end{align*}
The last equality derives from Remark \ref{rem:wvector}: $$\mathbb L^{-|w|}T^{m(w)}=\mathbb L^{-\sum_{i=1}^d\frac{\delta_{\{1,\ldots,d\}}}{\delta_i}}T^{\delta_{\{1,\ldots,d\}}}$$

We can do the same for the formula of Lemma \ref{lem:Bf} by taking the Euclidean division of $m$ by $\delta_{\{1,\ldots,d\}}$:
\begin{align*}
\sum_{m\ge1}\mathbb L^{-\sum_{i=1}^d\left\lfloor\frac{m}{\delta_i}\right\rfloor}T^m&=\sum_{q\ge0}\sum_{s=0}^{\delta_{\{1,\ldots,d\}}-1}\mathbb L^{-\sum_{i=1}^d\left(q\frac{\delta_{\{1,\ldots,d\}}}{\delta_i}+\left\lfloor\frac{s}{\delta_i}\right\rfloor\right)}T^{q\delta_{\{1,\ldots,d\}}+s}-\mathbb1 \\
&=\frac{\sum_{s=1}^{\delta_{\{1,\ldots,d\}}}\mathbb L^{-\sum_{i=1}^d\left\lfloor\frac{s}{\delta_i}\right\rfloor}T^{s}}{\mathbb 1-\mathbb L^{-|w|}T^{m(w)}}
\end{align*}

Hence, we obtain the following formula.
\begin{prop}
Let $f\in\mathbb R[x_1,\ldots,x_d]$ be a convenient non-degenerate weighted homogeneous polynomial. We denote by $\delta_i$ the exponent of the pure monomial $x_i$ of $f$. Then
\begin{align*}
\tilde Z_f(T)=\sum_{I\subset\{1,\ldots,d\}}\left(\left[f_{\tau_I}:(\mathbb R^*)^{|I|}\setminus f_{\tau_I}^{-1}(0)\rightarrow\mathbb R^*\right]-\left[\pr_2:f_{\tau_I}^{-1}(0)\cap(\mathbb R^*)^{|I|}\times\mathbb R^*\rightarrow\mathbb R^*\right]\right)\frac{\sum_{s=1}^{\frac{\delta_{\{1,\ldots,d\}}}{\delta_I}}\mathbb L^{-\sum_{i=1}^d\left\lfloor\frac{s\delta_I}{\delta_i}\right\rfloor}T^{s\delta_I}}{\mathbb 1-\mathbb L^{-|w|}T^{m(w)}} \\
\end{align*}
with $(\mathbb R^*)^{|\varnothing|}=\{0\}$, $f_{\tau_\varnothing}=0$ and $\delta_\emptyset=1$.
\end{prop}

\subsection{Proof of the invariance of the weigths}
\begin{lemma}\label{lem:exp2}
Let $f\in\mathbb R[x_1,\ldots,x_d]$ be a convenient non-degenerate weighted homogeneous polynomial. Denote by $a_n$ the coefficients of $\tilde Z_f(T)$, so that $$\tilde Z_f(T)=\sum_{n\ge1}a_nT^n$$ Denote by $\delta_i$ the exponent of the pure monomial $x_i$ of $f$. Assume that $\forall i,\delta_i\ge2$. Let $m$ be such that $2|m$ and $\delta_i\nmid m$ for $\delta_i\neq2$. Then $$a_m=\alpha_2\mathbb L^{-\sum_{i=1}^d\left\lfloor\frac{m}{\delta_i}\right\rfloor}$$ with $\beta F^+(\alpha_2)\neq0$.
\end{lemma}
\begin{proof}
Let $I_2=\{i,\delta_i=2\}\subset\{1,\ldots,d\}$ then, by Lemma \ref{lem:Bf} and Corollary \ref{cor:ComputeAf},
$$a_m=\alpha_2\mathbb L^{-\sum_{i=1}^d\left\lfloor\frac{m}{\delta_i}\right\rfloor}$$

where $$\alpha_2=\left(\sum_{\varnothing\neq I\subset I_2}\left(\left[f_{\tau_I}:(\mathbb R^*)^{|I|}\setminus f_{\tau_I}^{-1}(0)\rightarrow\mathbb R^*\right]-\left[\pr_2:f_{\tau_I}^{-1}(0)\cap(\mathbb R^*)^{|I|}\times\mathbb R^*\rightarrow\mathbb R^*\right]\right)\right)-\mathbb 1$$

We may assume that $I_2\neq\varnothing$ since, otherwise, $\beta F^+(\alpha_2)=0-1\neq0$.

Next,
\begin{align*}
\beta F^+(\alpha_2)&=\left(\sum_{\varnothing\neq I\subset I_2}\left(\beta\left(f_{\tau_{I}}^{-1}(1)\subset(\mathbb R^*)^{|I|}\right)-\beta\left(f_{\tau_{I}}^{-1}(0)\subset(\mathbb R^*)^{|I|}\right)\right)\right)-1\\
&=\sum_{\varnothing\neq I\subset I_2}\beta\left(f_{\tau_{I}}^{-1}(1)\subset(\mathbb R^*)^{|I|}\right)-\sum_{\varnothing\neq I\subset I_2}\beta\left(f_{\tau_{I}}^{-1}(0)\subset(\mathbb R^*)^{|I|}\right)-\beta(\{0\})\\
&=\beta\left(f_{\tau_{I_2}}^{-1}(1)\subset\mathbb R^{|I_2|}\right)-\beta\left(f_{\tau_{I_2}}^{-1}(0)\subset\mathbb R^{|I_2|}\right)
\end{align*}
For the last equality, we use the identification of $\mathbb R^{|I|}$ with $\left\{(x_1,\ldots,x_d)\in\mathbb R^d,\,\forall j\notin I,x_j=0\right\}$. Indeed, for $I\subset I_2$, we have $\mathbb R^{|I|}\subset\mathbb R^{|I_2|}$ and so $$\beta\left(f_{\tau_{I}}^{-1}(1)\subset(\mathbb R^*)^{|I|}\right)=\beta\left(x\in\mathbb R^{|I_2|},f_{\tau_{I_2}}(x)=1,\forall i\in I, x_i\neq0,\forall j\in I_2\setminus I, x_j=0\right)$$
hence we may rewrite the first sum using the additivity of $\beta$:
\begin{align*}
\sum_{\varnothing\neq I\subset I_2}\beta\left(f_{\tau_{I}}^{-1}(1)\subset(\mathbb R^*)^{|I|}\right)&=\sum_{\varnothing\neq I\subset I_2}\beta\left(x\in\mathbb R^{|I_2|},f_{\tau_{I_2}}(x)=1,\forall i\in I, x_i\neq0,\forall j\in I_2\setminus I, x_j=0\right) \\
&=\beta\left(\bigsqcup_{\varnothing\neq I\subset I_2}\left\{x\in\mathbb R^{|I_2|},f_{\tau_{I_2}}(x)=1,\forall i\in I, x_i\neq0,\forall j\in I_2\setminus I, x_j=0\right\}\right) \\
&=\beta\left(x\in\mathbb R^{|I_2|},f_{\tau_{I_2}}(x)=1\right)
\end{align*}
We use the exact same argument for the second sum with the additional term $\beta(\{0\})$ since $0\in f_{\tau_{I_2}}^{-1}(0)\subset\mathbb R^{|I_2|}$. \\

By a linear change of coordinates, we may assume that $f_{\tau_{I_2}}$ is a homogeneous Brieskorn polynomial of degree 2. When we take the value at $u=-1$, which is just the Euler characteristic with compact support, the first term is odd by \cite[Proposition 2.1]{Fic06} or \cite[Proposition 2.1]{Fic06-rims} whereas the second term is even by \cite[Corollary 2.5]{Fic06} or \cite[Proposition 2.1]{Fic06-rims}. Hence $\beta F^+(\alpha_2)(u=-1)\neq0$ and so $\beta F^+(\alpha_2)\neq0$.
\end{proof}

\begin{rem}
The proof of Lemma \ref{lem:exp4} works as it is to prove Lemma \ref{lem:exp2}. Nevertheless, the above proof of Lemma \ref{lem:exp2} is easier for that particular situation.
\end{rem}

\begin{lemma}\label{lem:exp4}
Let $f\in\mathbb R[x_1,\ldots,x_d]$ be a convenient non-degenerate weighted homogeneous polynomial. Denote by $a_n$ the coefficients of $\tilde Z_f(T)$, so that $$\tilde Z_f(T)=\sum_{n\ge1}a_nT^n$$ Denote by $\delta_i$ the exponent of the pure monomial $x_i$ of $f$. Assume that $\forall i,\delta_i\ge2$. Let $m$ be such that $4|m$ and $\delta_i\nmid m$ for $\delta_i\neq2,4$. Then $$a_m=\alpha_4\mathbb L^{-\sum_{i=1}^d\left\lfloor\frac{m}{\delta_i}\right\rfloor}$$ with $\beta F^+(\alpha_4)\neq0$.
\end{lemma}
\begin{proof}
Let $I_4=\{i,\delta_i|4\}\subset\{1,\ldots,d\}$ then, by Lemma \ref{lem:Bf} and Corollary \ref{cor:ComputeAf},
$$a_m=\alpha_4\mathbb L^{-\sum_{i=1}^d\left\lfloor\frac{m}{\delta_i}\right\rfloor}$$

where $$\alpha_4=\left(\sum_{\varnothing\neq I\subset I_4}\left(\left[f_{\tau_I}:(\mathbb R^*)^{|I|}\setminus f_{\tau_I}^{-1}(0)\rightarrow\mathbb R^*\right]-\left[\pr_2:f_{\tau_I}^{-1}(0)\cap(\mathbb R^*)^{|I|}\times\mathbb R^*\rightarrow\mathbb R^*\right]\right)\right)-\mathbb 1$$

We may assume that $I_4\neq\varnothing$ since, otherwise, $\beta F^+(\alpha_4)=0-1\neq0$.

Next
\begin{align*}
\chi_c F^+(\alpha_4)&=\left(\sum_{\varnothing\neq I\subset I_4}\left(\chi_c\left(f_{\tau_{I}}^{-1}(1)\subset(\mathbb R^*)^{|I|}\right)-\chi_c\left(f_{\tau_{I}}^{-1}(0)\subset(\mathbb R^*)^{|I|}\right)\right)\right)-1\\
&=\sum_{\varnothing\neq I\subset I_4}\chi_c\left(f_{\tau_{I}}^{-1}(1)\subset(\mathbb R^*)^{|I|}\right)-\sum_{\varnothing\neq I\subset I_4}\chi_c\left(f_{\tau_{I}}^{-1}(0)\subset(\mathbb R^*)^{|I|}\right)-\chi_c(\{0\})\\
&=\chi_c\left(f_{\tau_{I_4}}^{-1}(1)\subset\mathbb R^{|I_4|}\right)-\chi_c\left(f_{\tau_{I_4}}^{-1}(0)\subset\mathbb R^{|I_4|}\right)
\end{align*}
For the last equality, we use the identification of $\mathbb R^{|I|}$ with $\left\{(x_1,\ldots,x_d)\in\mathbb R^d,\,\forall j\notin I,x_j=0\right\}$. Indeed, for $I\subset I_2$, we have $\mathbb R^{|I|}\subset\mathbb R^{|I_2|}$ and so $$\chi_c\left(f_{\tau_{I}}^{-1}(1)\subset(\mathbb R^*)^{|I|}\right)=\chi_c\left(x\in\mathbb R^{|I_4|},f_{\tau_{I_4}}(x)=1,\forall i\in I, x_i\neq0,\forall j\in I_4\setminus I, x_j=0\right)$$
hence we may rewrite the first sum using the additivity of $\chi_c$ as in Lemma \ref{lem:exp2}. \\
We use the exact same argument for the second sum with the additional term $\chi_c(\{0\})$ since $0\in f_{\tau_{I_4}}^{-1}(0)\subset\mathbb R^{|I_4|}$. \\

Denote by $f_{\tau_{I_4}}^\mathbb C$ the complexification of $f_{\tau_{I_4}}$. Since $$\chi_c\left(f_{\tau_{I_4}}^{-1}(x)\right)\equiv\chi\left(f_{\tau_{I_4}}^{\mathbb C-1}(x)\right)\mod 2$$ we have $$\chi_c\left(f_{\tau_{I_4}}^{-1}(1)\subset\mathbb R^{|I_4|}\right)-\chi_c\left(f_{\tau_{I_4}}^{-1}(0)\subset\mathbb R^{|I_4|}\right)\equiv\chi\left(f_{\tau_{I_4}}^{\mathbb C-1}(1)\subset\mathbb C^{|I_4|}\right)-\chi\left(f_{\tau_{I_4}}^{\mathbb C-1}(0)\subset\mathbb C^{|I_4|}\right)\mod2$$

Since $f_{\tau_{I_4}}^\mathbb C$ is weighted homogeneous, $f_{\tau_{I_4}}^{\mathbb C-1}(1)$ is homeomorphic to the Milnor fiber of $f_{\tau_{I_4}}^{\mathbb C}$ at $0$ \cite[Lemma 9.4]{Mil68} and since $f_{\tau_{I_4}}^\mathbb C$ is convenient and non-degenerate, the Euler characteristic of its Milnor fiber can be computed from its Newton polyhedron (See \cite[p. 64]{Mil68} and \cite[1.10 Théorème I.(ii).]{Kou76}).

Since $f_{\tau_{I_4}}^\mathbb C$ is non-degenerate, $\chi_c\left(f_{\tau_{I_4}}^{\mathbb C-1}(0)\subset\mathbb C^{|I_4|}\right)$ can be computed from its Newton polyhedron formula \cite[Theorem 2]{Hov78}.

Hence $$\chi_c\left(f_{\tau_{I_4}}^{-1}(1)\subset\mathbb R^{|I_4|}\right)-\chi_c\left(f_{\tau_{I_4}}^{-1}(0)\subset\mathbb R^{|I_4|}\right)\mod2$$ only depends on the Newton polyhedron of $f_{\tau_{I_4}}$.

So we may assume that $$f_{\tau_{I_4}}(x)=\sum_i\varepsilon_ix_i^2+\sum_j\eta_jx_i^4$$ with $\varepsilon_i,\eta_j\in\{\pm1\}$. But then (see the proof of \cite[Proposition 8.3]{JBC2}) $$\chi_c\left(f_{\tau_{I_4}}^{-1}(1)\subset\mathbb R^{|I_4|}\right)-\chi_c\left(f_{\tau_{I_4}}^{-1}(0)\subset\mathbb R^{|I_4|}\right)\equiv1\mod2$$

Thereby $\beta F^+(\alpha_4)(u=-1)=\chi_c F^+(\alpha_4)\neq0$ and thus $\beta F^+(\alpha_4)\neq0$.
\end{proof}

The previous results allow one to prove Theorem \ref{thm:Main} using the proof of \cite[Proposition 8.3]{JBC2}. We recall it for ease of reading.

The proof is structured as follows. We first handle the non-singular case so that we can next focus on the singular one. In the latter, the idea is to prove that we can recover the weights of a singular convenient non-degenerate weighted homogeneous polynomial $f$ from its modified zeta function. Since it is an invariant of the arc-analytic equivalence, this induces that the arc-analytic type of such a polynomial determines its weights.

More precisely we are going to deduce the exponents $\delta_i$ of the pure monomials of $f$ from its modified zeta function, which is enough by Remark \ref{rem:wvector}. The first step is to find an upper bound $K$ of these exponents in order to reduce to a finite number of possibilities. Finally, we construct (and solve) a linear system in the unknown variables $\mult(q)=\#\{i,\,\delta_i=q\}$ for $q\le K$, which is enough to recover the exponents.

\begin{proof}[Proof of Theorem \ref{thm:Main}]
We first fix some notations. \\
Let $f\in\mathbb R[x_1,\ldots,x_d]$ be a convenient non-degenerate weighted homogeneous polynomial. Denote by $a_n$ the coefficients of $\tilde Z_f(T)$, so that $$\tilde Z_f(T)=\sum_{n\ge1}a_nT^n$$ Denote by $\delta_i$ the exponent of the pure monomial $x_i$ of $f$. \\

Notice that $f$ is singular if and only if $\forall i,\delta_i\ge2$. Indeed, if, for example, $\delta_1=1$ then $f(x_1,\ldots,x_d)=x_1+h(x_2,\ldots,x_d)$ since the monomials of $f$ lie in the convex hull of $x_i^{\delta_i}$ and thus $f$ is non-singular. Otherwise, if $\forall i,\delta_i\ge2$ then we may apply the Euler formula to the partial derivatives of $f$ to show that $f$ is singular at the origin.

If $f$ is non-singular, then, by a change of coordinates, we may assume $\delta_1=1$ and $f$ is arc-analytically equivalent to $(x_1,\ldots,x_d)\mapsto x_1$ by applying the Nash inverse function theorem to $(x_1,\ldots,x_d)\mapsto\left(f(x_1,\ldots,x_d),x_2,\ldots,x_d\right)$. Notice that in this case $\tilde Z_f(T)=0$. Actually, by Corollary \ref{cor:nonmult}, $f$ is non-singular if and only if $\tilde Z_f(T)=0$.

Now let $f$ and $g$ be two convenient non-degenerate weighted homogeneous polynomials which are arc-analytically equivalent. If $f$ is non-singular then $\tilde Z_f(T)=\tilde Z_g(T)=0$ and $g$ is also non-singular. Moreover, they are both arc-analytically equivalent to $(x_1,\ldots,x_d)\mapsto x_1$. \\

Hence, from now on, we may assume that $\forall i\in\{1,\ldots,d\},\,\delta_i\ge2$. \\

Then, by Corollary \ref{cor:nonmult}, $$\lim_{p}\frac{1-\deg\beta\left(\overline{a_p}\right)}{p}=\lim_{p}\frac{\sum_{i=1}^d\left\lfloor\frac{p}{\delta_i}\right\rfloor}{p}=\sum_{i=1}^d\frac{1}{\delta_i}$$ where $p$ goes through the prime numbers.

Hence, we may deduce from \cite[Lemma 8.5]{JBC2} an upper bound $K$ of the set $\{\delta_1,\ldots,\delta_d\}$. \\

For $q\in\mathbb N$, set $\mult(q)=\#\{i,\,\delta_i=q\}$. The goal is to prove that we can retrieve $\mult(q)$ for $q\le K$ from $\tilde Z_f(T)$. This will prove, using Remark \ref{rem:wvector}, that we may recover the weights of $f$ from $\tilde Z_f(T)$. Which is enough to achieve the proof of Theorem \ref{thm:Main} since $\tilde Z_f(T)$ is an invariant of the arc-analytic equivalence. \\

First, to lighten the redaction, we enlarge the set $\{1,\ldots,K\}$ as follows. Set $$\mathcal P=\{p\text{ prime},\,p\le K\}$$ and for $p\in\mathcal P$, set $$\gamma_p=\max\left\{\gamma,\,p^\gamma\le K\right\}$$ Set $$Q=\left\{\prod_{p\in\mathcal P}p^{\alpha_p},\,0\le\alpha_p\le\gamma_p\right\}$$ Then $\{\delta_1,\ldots,\delta_d\}\subset\{1,\ldots,K\}\subset Q$ and we are going to compute $\mult(q)$ for $q\in Q$ from the coefficients of $\tilde Z_f(T)$. \\

Notice that $\mult1=0$. \\

The proof is now divided into different steps.

\begin{enumerate}[label=\underline{Step \arabic*:},ref=Step \arabic*]
\item\label{item:six} Equations involving $\mult(q)$ for $q\in Q$ such that $6|q$. \\
Let $q=\prod_{p\in\mathcal P}p^{\alpha_p}$ with $\alpha_2,\alpha_3\ge1$ and $0\le\alpha_p\le\gamma_p$ for $p\in\mathcal P\setminus\{2,3\}$. By the Chinese Remainder Theorem, we may find $n\in\mathbb N\setminus\{0,1\}$ such that:
$$\left\{\begin{array}{l}
\forall r\in Q,\,r|n-1\Leftrightarrow r=1\\
\forall r\in Q,\,r|n\Leftrightarrow r|q\\
\forall r\in Q,\,r|n+1\Leftrightarrow r=1
\end{array}\right.$$

Indeed, it suffices to choose $n$ such that $$\left\{\begin{array}{ll}n\equiv p^{\alpha_p}\mod p^{\alpha_p+1}&\text{ if $\alpha_p\ge1$}\\n\equiv2\mod p&\text{ if $\alpha_p=0$}\end{array}\right.$$

Hence, by Corollary \ref{cor:nonmult}, $$-\deg\beta(\overline{a_{n+1}})+\deg\beta(\overline{a_{n-1}})=\sum_{\substack{r\in Q\\r|q}}\mult(r)$$

\item\label{item:mult2} Computation of $\mult 2$. \\
Similarly, let $n$ be such that \\
$\left\{\begin{array}{ll}
n\equiv2^2\mod2^3&\\
n\equiv3\mod3^2&\\
n\equiv5\mod5^2&\\
n\equiv3\mod p\text{ for $p\in\mathcal P\setminus\{2,3,5\}$}
\end{array}\right.$
then\quad\quad
$\left\{\begin{array}{l}
\forall r\in Q,\,r|n-2\Leftrightarrow r|2\\
\forall r\in Q,\,r|n-1\Leftrightarrow r=1\\
\forall r\in Q,\,r|n\Leftrightarrow r|2^2\cdot3\cdot5\\
\forall r\in Q,\,r|n+1\Leftrightarrow r=1\\
\forall r\in Q,\,r|n+2\Leftrightarrow r|2
\end{array}\right.$

Hence, by Lemma \ref{lem:exp2}, $$-\deg\beta F^+(a_{n+2})+\deg\beta F^+(a_{n-2})=\mult2+\sum_{\substack{r\in Q\\r|n}}\mult(r)$$

We may derive $\mult 2$ from the last since, in \ref{item:six}, we got an equation of the form $$\sum_{\substack{r\in Q\\r|n}}\mult(r)=\cst$$ where the right-hand side ``$\cst$'' can be computed in terms of the $\beta (\overline a_i)$'s.

\item\label{item:no23} Computation of $\mult q$ for $q\in Q$ such that $2\nmid q,3\nmid q$. \\
Let $q=\prod_{p\in\mathcal P}p^{\alpha_p}$ with $\alpha_2=\alpha_3=0$ and $0\le\alpha_p\le\gamma_p$ for $p\in\mathcal P\setminus\{2,3\}$. \\
Similarly, let $n$ be such that \\
$\left\{\begin{array}{ll}
n\equiv p^{\alpha_p}\mod p^{\alpha_p+1} & \text{ if $\alpha_p\ge1$} \\
n-1\equiv0\mod2^3&\\
n-1\equiv0\mod p&\text{ if $p\neq2$ and $\alpha_p=0$}
\end{array}\right.$
then\quad\quad
$\left\{\begin{array}{l}
\forall r\in Q,\,r|n-3\Leftrightarrow r|2\\
\forall r\in Q,\,r|n-2\Leftrightarrow r=1\\
2\cdot3|n-1\\
\forall r\in Q,\,r|n\Leftrightarrow r|q\\
\forall r\in Q,\,r|n+1\Leftrightarrow r|2
\end{array}\right.$

Hence, by Lemma \ref{lem:exp2}, $$-\deg\beta F^+(a_{n+1})+\deg\beta F^+(a_{n-3})=\mult2+\sum_{\substack{r\in Q\\r|n-1}}\mult(r)+\sum_{\substack{r\in Q\\r|q}}\mult(r)$$

Since we already computed $\mult2$ in \ref{item:mult2} and since we already got, in \ref{item:six}, an equation of the form $$\sum_{\substack{r\in Q\\r|n-1}}\mult(r)=\cst$$ we derive an equation of the form $$\sum_{\substack{r\in Q\\r|q}}\mult(r)=\cst$$

Then we may compute recursively $\mult q$ for $q\in Q$ such that $2\nmid q,3\nmid q$ by varying the $\alpha_p$.

\item\label{item:3et4} Computation of $\mult3$ and $\mult4$. \\
Similarly, let $n$ be such that\\
$\left\{\begin{array}{ll}
n\equiv2^2\mod2^3&\\
n+1\equiv5\mod5^2&\\
n+2\equiv6\mod3^2&\\
n+2\equiv0\mod p&\text{ if $p\in\mathcal P\setminus\{2,3,5\}$}
\end{array}\right.$
then\quad\quad
$\left\{\begin{array}{l}
\forall r\in Q,\,r|n-3\Leftrightarrow r=1\\
\forall r\in Q,\,r|n-2\Leftrightarrow r|2\\
\forall r\in Q,\,r|n-1\Leftrightarrow r|3\\
\forall r\in Q,\,r|n\Leftrightarrow r|2^2\\
\forall r\in Q,\,r|n+1\Leftrightarrow r|5\\
2\cdot3|n+2\\
\forall r\in Q,\,r|n+3\Leftrightarrow r=1\\
\end{array}\right.$

Hence, by Corollary \ref{cor:nonmult}, $$-\deg\beta F^+(a_{n+3})+\deg\beta F^+(a_{n-3})=2\mult2+\mult3+\mult4+\mult5+\sum_{\substack{r\in Q\\r|n+2}}\mult(r)$$

Since we computed $\mult 2$ in \ref{item:mult2}, since we computed $\mult 5$ in \ref{item:no23} and since we get an equation of the form $\sum_{\substack{r\in Q\\r|n+2}}\mult(r)=\cst$ in \ref{item:six}, we get an equation of the form 
\begin{equation}\label{eqn:3et4:1}\mult3+\mult4=\cst\end{equation}

Now let $n$ be such that \\
$\left\{\begin{array}{ll}
n\equiv2^3\mod2^4&\\
n\equiv3^2\mod3^3&\\
n\equiv0\mod p&\text{ if $p\in\mathcal P\setminus\{2,3\}$}
\end{array}\right.$
then\quad
$\left\{\begin{array}{l}
\forall r\in Q,\,r|n-4\Leftrightarrow r|2^2\\
\forall r\in Q,\,r|n-3\Leftrightarrow r|3\\
\forall r\in Q,\,r|n-2\Leftrightarrow r|2\\
\forall r\in Q,\,r|n-1\Leftrightarrow r=1\\
2\cdot3|n\\
\forall r\in Q,\,r|n+1\Leftrightarrow r=1\\
\forall r\in Q,\,r|n+2\Leftrightarrow r|2\\
\forall r\in Q,\,r|n+3\Leftrightarrow r|3\\
\forall r\in Q,\,r|n+4\Leftrightarrow r|2^2\\
\end{array}\right.$

Hence, by Lemma \ref{lem:exp4}, $$-\deg\beta F^+(a_{n+4})+\deg\beta F^+(a_{n-4})=\sum_{\substack{r\in Q\\r|n}}\mult(r)+3\mult2+2\mult3+\mult4$$

Since we computed $\mult 2$ in \ref{item:mult2} and since we get an equation of the form $\sum_{\substack{r\in Q\\r|n}}\mult(r)=\cst$ in \ref{item:six}, we get an equation of the form 
\begin{equation}\label{eqn:3et4:2}2\mult3+\mult4=\cst\end{equation}

Equations \eqref{eqn:3et4:1} and \eqref{eqn:3et4:2} allow us to compute $\mult3$ and $\mult4$.

\item\label{item:pas4pas3} Computation of $\mult q$ for $q\in Q$ such that $2|q,3\nmid q,4\nmid q$. \\
Let $q=\prod_{p\in\mathcal P}p^{\alpha_p}$ with $\alpha_2=1$, $\alpha_3=0$ and $0\le\alpha_p\le\gamma_p$ for $p\in\mathcal P\setminus\{2,3\}$. Similarly, let $n$ be such that

$\left\{\begin{array}{ll}
n\equiv p^{\alpha_p}\mod p^{\alpha_p+1} & \text{ if $\alpha_p\ge1$ and $p\neq2$} \\
n+1\equiv5\mod5^2&\text{ if $\alpha_5=0$}\\
n+2\equiv2^3\mod2^4&\\
n+2\equiv6\mod3^2&\\
n+2\equiv0\mod7&\text{ if $\alpha_7=0$}\\
n\equiv4\mod p&\text{ for $p\in\mathcal P\setminus\{2,3,5,7\}$, $\alpha_p=0$}
\end{array}\right.$
then\quad
$\left\{\begin{array}{l}
\forall r\in Q,\,r|n-3\Leftrightarrow r=1\\
\forall r\in Q,\,r|n-2\Leftrightarrow r|2^2\\
\forall r\in Q,\,r|n-1\Leftrightarrow r|3\\
\forall r\in Q,\,r|n\Leftrightarrow r|q\\
\forall r\in Q,\,r|n+1\Leftrightarrow r=1 \text{ if $\alpha_5\ge1$}\\
\forall r\in Q,\,r|n+1\Leftrightarrow r|5 \text{ if $\alpha_5=0$}\\
2\cdot3|n+2\\
\forall r\in Q,\,r|n+3\Leftrightarrow r=1
\end{array}\right.$

Hence, by Corollary \ref{cor:nonmult}, $$-\deg\beta(\overline{a_{n+3}})+\deg\beta(\overline{a_{n-3}})=\sum_{\substack{r\in Q\\r|q}}\mult(r)+\sum_{\substack{r\in Q\\r|n+2}}\mult(r)+\mult2+\mult3+\mult4(+\mult5)$$

Since we already computed $\mult 2$ in \ref{item:mult2}, $\mult3,\mult4$ in \ref{item:3et4}, $\mult5$ in \ref{item:no23} and since we got in \ref{item:six} an equation of the form $$\sum_{\substack{r\in Q\\r|n+2}}\mult(r)=\cst$$ we finally obtain an equation of the form $$\sum_{\substack{r\in Q\\r|q}}\mult(r)=\cst$$

Thus we may compute $\mult q$ for $q\in Q$ such that $2|q,3\nmid q,4\nmid q$ by varying the $\alpha_p$.

\item\label{item:4pas3} Computation of $\mult q$ for $q\in Q$ such that $4|q,3\nmid q$. \\
Let $q=\prod_{p\in\mathcal P}p^{\alpha_p}$ with $\alpha_2\ge2$, $\alpha_3=0$ and $0\le\alpha_p\le\gamma_p$ for $p\in\mathcal P\setminus\{2,3\}$. Similarly, let $n$ be such that \\
$\left\{\begin{array}{ll}
n\equiv p^{\alpha_p}\mod p^{\alpha_p+1}&\text{ if $\alpha_p\ge1$} \\
n-1\equiv3\mod3^2&\\
n-2\equiv5\mod5^2&\text{ if $\alpha_5=0$}\\
n\equiv4\mod p&\text{ for $p\in\mathcal P\setminus\{2,3,5\}$, $\alpha_p=0$}
\end{array}\right.$
then\quad
$\left\{\begin{array}{l}
\forall r\in Q,\,r|n-3\Leftrightarrow r=1\\
\forall r\in Q,\,r|n-2\Leftrightarrow r|2\text{ if $\alpha_5\ge1$}\\
\forall r\in Q,\,r|n-2\Leftrightarrow r|2\cdot5\text{ if $\alpha_5=0$}\\
\forall r\in Q,\,r|n-1\Leftrightarrow r|3\\
\forall r\in Q,\,r|n\Leftrightarrow r|q\\
\forall r\in Q,\,r|n+1\Leftrightarrow r|1
\end{array}\right.$

Hence, by Corollary \ref{cor:nonmult}, $$-\deg\beta(\overline{a_{n+1}})+\deg\beta(\overline{a_{n-3}})=\sum_{\substack{r\in Q\\r|q}}\mult(r)+\sum_{\substack{r\in Q\\r|n-2}}\mult(r)+\mult3$$

Since we already computed $\mult3$ in \ref{item:3et4} and since we got an equation of the form $\sum_{\substack{r\in Q\\r|n-2}}\mult(r)$ in \ref{item:pas4pas3}, we obtain an equation of the form $$\sum_{\substack{r\in Q\\r|q}}\mult(r)=\cst$$

Thus we may compute $\mult q$ for $q\in Q$ such that $4|q,3\nmid q$ by varying the $\alpha_p$.

\item Computation of $\mult q$ for $q\in Q$ such that $3|q,2\nmid q$. \\
Let $q=\prod_{p\in\mathcal P}p^{\alpha_p}$ with $\alpha_3\ge1$, $\alpha_2=0$ and $0\le\alpha_p\le\gamma_p$ for $p\in\mathcal P\setminus\{2,3\}$. Similarly, let $n$ be such that \\
$\left\{\begin{array}{ll}
n-1\equiv2^2\mod2^3&\\
n\equiv p^{\alpha_p}\mod p^{\alpha_p+1}&\text{ if $\alpha_p\ge1$} \\
n+1\equiv5\mod5^2&\text{ if $\alpha_5=0$}\\
n\equiv3\mod p&\text{ for $p\in\mathcal P\setminus\{2,3,5\}$, $\alpha_p=0$}
\end{array}\right.$
then\quad
$\left\{\begin{array}{l}
\forall r\in Q,\,r|n-2\Leftrightarrow r|1\\
\forall r\in Q,\,r|n-1\Leftrightarrow r|2^2\\
\forall r\in Q,\,r|n\Leftrightarrow r|q\\
\forall r\in Q,\,r|n+1\Leftrightarrow r|2\text{ if $\alpha_5\ge1$}\\
\forall r\in Q,\,r|n+1\Leftrightarrow r|2\cdot5\text{ if $\alpha_5=0$}\\
\forall r\in Q,\,r|n+2\Leftrightarrow r|1
\end{array}\right.$

Hence, by Corollary \ref{cor:nonmult}, $$-\deg\beta(\overline{a_{n+2}})+\deg\beta(\overline{a_{n-2}})=\sum_{\substack{r\in Q\\r|q}}\mult(r)+\sum_{\substack{r\in Q\\r|n-1}}\mult(r)+\sum_{\substack{r\in Q\\r|n+1}}\mult(r)$$

Since we got an equation of the form $\sum_{\substack{r\in Q\\r|n-1}}\mult(r)$ in \ref{item:4pas3} and an equation of the form $\sum_{\substack{r\in Q\\r|n+1}}\mult(r)$ in \ref{item:pas4pas3}, we obtain an equation of the form $$\sum_{\substack{r\in Q\\r|q}}\mult(r)=\cst$$

Thus we may compute $\mult q$ for $q\in Q$ such that $3|q,2\nmid q$ by varying the $\alpha_p$.

\item Computation of $\mult(q)$ for $q\in Q$ such that $6|q$. \\
We may now compute  $\mult(q)$ for $q\in Q$ such that $6|q$ by varying the $\alpha_p$ in \ref{item:six}.
\end{enumerate}
\end{proof}

The arguments of the previous proof also allow one to prove \cite[Corollary 8.4]{JBC2} without using the convolution formula.
\begin{thm}[{\cite[Corollary 8.4]{JBC2}}]
The arc-analytic type of a singular Brieskorn polynomial determines its exponents.
\end{thm}

\footnotesize
\def\cprime{$'$}

\end{document}